\documentclass[12pt,oneside,reqno]{amsart}
\usepackage[utf8]{inputenc}
\usepackage{amssymb}
\usepackage{geometry}
\usepackage{amsmath,amscd}
\usepackage{graphicx}
\usepackage{color,hyperref}
\usepackage{tikz}
\usepackage{tikz-cd}
\usepackage{amsfonts}
\usetikzlibrary{matrix,arrows}
\usepackage{mathrsfs}
\usepackage{amssymb}
\usepackage{adjustbox}
\usepackage{physics}
\begin{document}
	
	\newcommand{\mf}{\mathfrak}
	\newcommand{\mc}{\mathcal}
	\newcommand{\mb}{\mathbf}
	
	\newcommand{\R}{\mathbf R}
	\newcommand{\C}{\mathbf C}
	\newcommand{\Q}{\mathbf Q}
	\newcommand{\Z}{\mathbf Z}
	\newcommand{\F}{\mathbf F}
	\newcommand{\N}{\mathbf N}

	\newcommand{\Fix}{\textnormal{Fix}}
	\newcommand{\End}{\textnormal{End}}
	\newcommand{\Frob}{\textnormal{Frob}}

	\newcommand{\sm}[4]{{
			\left(\begin{smallmatrix}{#1}&{#2}\\{#3}&{#4}\end{smallmatrix}\right)}}
	\newcommand{\sv}[2]{
		\genfrac(){0pt}{1}{#1}{#2}}

	\numberwithin{equation}{section}
	
	\theoremstyle{plain}
	\newtheorem{theorem}{Theorem}
	\newtheorem{lemma}[theorem]{Lemma}
	\newtheorem{proposition}[theorem]{Proposition}
	\newtheorem{corollary}[theorem]{Corollary}
	\newtheorem*{t2}{Theorem 2}
	\newtheorem*{p1}{Proposition 1}
	
	\theoremstyle{definition}
	\newtheorem{definition}[theorem]{Definition}
	\newtheorem{example}[theorem]{Example}
	\newtheorem{conjecture}[theorem]{Conjecture}
	\newtheorem{remark}[theorem]{Remark}

	\title{On the p-rank of singular curves and their smooth models}
	\author{SADIK TERZ\.{I}}
	
	\address{Middle East Technical University (NCC), Mathematics Department, Kalkanli, KKTC, Mersin 10, Turkey.}
	\email{terzi@metu.edu.tr}
	
	
	
	
	\maketitle
	
	\begin{abstract}
		In this paper, we are concerned with the computation of the $p$-rank and $a$-number of singular curves and their smooth model. We consider a pair $X, X'$ of proper curves over an algebraically closed field $k$ of characteristic $p$, where $X'$ is a singular curve which lies on a smooth projective variety, particularly on smooth projective surfaces $S$ (with $p_g(S) = 0 = q(S)$) and $X$ is the smooth model of $X'$. We determine the $p$-rank of $X$ by using the exact sequence of group schemes relating the Jacobians $J_X$ and $J_{X'}$. As an application, we determine a relation about the fundamental invariants $p$-rank and $a$-number of a family of singular curves and their smooth models. Moreover, we calculate $a$-number and find lower bound for $p$-rank of a family of smooth curves.  
	\end{abstract}
	
	\vskip 0.5cm
	\section{Introduction}
Let $X$ be a smooth projective  curve of genus $g\geq 2$ over an algebraically closed field $k$ of  characteristic  $p > 0$ and $J_X$ be its Jacobian. The $p$-rank $\sigma(X)$ and $a$-number $a(X)$ are fundamental invariants of $X$(see Definition \ref{od} below), and have been studied extensively by determining the action of the Frobenius map on the cohomology group $H^1(X, \mc{O}_X)$ or equivalently the action of the Cartier operator on $H^0(X, \Omega_X)$. In the former case one essentially determines the Hasse-Witt matrix (\cite{M}) and in the latter the Cartier-Manin matrix (\cite{Y})  describing the action. For a more extensive bibliography on Hasse-Witt and Cartier-Manin matrices we refer to \cite{A}.  We are interested in determining the $p$-rank $\sigma (X)$ and $a$-number $a(X)$ for curves in certain varieties. 
	\par If $\pi : X \to X'$ is the resolution of singularities of a curve $X'$ lying on a smooth variety, especially on smooth projective surface $S$, then in principle one can determine  $\sigma(X)$ by computing
	the action of Frobenius on $H^1(X', \mc{O}_{X'})$ and relating this computation to $X$ via the cohomology 
	sequence attached to the resolution. In particular, if  $p_g(S) = 0 = q(S)$ where $p_g(S)  \text{ and } q(S)$ are the geometric genus and the irregularity of $S$ respectively, the method is quite effective because in this case the action of Frobenius can be easily calculated. We illustrated this property in \cite[Section 2, Ex.9]{st} for curves on Hirzebruch surfaces. As there are curves not defined on projective plane ${\mb P}^2_k$ but possibly defined on Hirzebruch surfaces, one can expect that constraining a curve in a specific ambient space and taking advantage of its geometry enables one to determine the $p$-rank and the $a$-number of the curve. In fact, the explicit computation of a basis for $H^1(X,\mathcal{O}_X) $ and the Frobenius map on $H^1(X,\mathcal{O}_X) $ will be useful to calculate $\sigma(X)$ and $a(X)$ for such curves $X$.\\
	
	In the calculations explained in the preceding paragraph we
	incorporate  the exact sequence 
	$$  0 \longrightarrow G \longrightarrow J_{X'} \longrightarrow J_X \longrightarrow 0 $$
	of group schemes arising from the resolution $X \to X'$.
	In the second section of the paper,  we discuss the  
	effect of the singularities of $X'$ on the structure of $G[p]$. Then once  
	we determine the structure of the subgroup $J_{X'}[p]$, we obtain the $p$-rank of the smooth curve $X$ and a lower bound for $a(X)$. More precisely, we have:
	
	\begin{proposition}{\label{p}} In the given setup, the following relations hold:
	\begin{itemize}
	    \item[1)] $\sigma(J_X) = \sigma(J_{X'}) - \sigma(G)$
	    \item[2)] $a(X) \ge a(J_{X'}) - a(G)$.
	\end{itemize}
	 \end{proposition} 
  \par In third section, we provide a family of examples (Ex. \ref{ge}) of complete intersection curves to find their $a$-number and lower bound on their $p$-rank by using explicit basis constructed in \cite[Theorem 2]{st} and by using explicit action of the Frobenius map computed in \cite[Proposition 14]{st}. This family consists of generalized Fermat curves $F_{m,n}$ of type $(m,n)$ and forms a family of algebraic curves of dimension $n-2$ in the moduli space of projective smooth genus $ g(F_{m,n}) = 1 + \frac{m^{n-1}}{2}((m-1)(n-1)-2)$ curves (\cite{Hi}, Section 2). The investigation of algebraic curves over fields of characteristic $p>0$ is related to several problems for curves over finite fields, such as the cardinality of the set of rational points, the search for maximal curves with respect to the Hasse-Weil
bound, properties on zeta functions and Weierstrass points on curves. Many results
have been obtained for classical Fermat curves (i.e., $n=2$) \cite{GV, VZ, N}.
 \par The algebraic model of generalized Fermat curve $F_{m,n}$ is as follows:
 $$ C^m(\lambda_0,\lambda_1,...,\lambda_{n-2}) := \begin{Bmatrix}  \lambda_0 x_0^m+x^m_1+x^m_2=0\\ \lambda_1 x_0^m+x^m_1+x^m_3=0 \\ \vdots \\ \lambda_{n-2}x_0^m+x^m_1+x^m_n=0\end{Bmatrix} \subset {\mb P}^n $$ where the constants $ \lambda_0,\lambda_1,...,\lambda_{n-2}$ are pairwise different and $\lambda_i \neq 0,1$. We set $X:= C^m(\lambda_0,\lambda_1,...,\lambda_{n-2})$ and define sets 
\[
   S(r,s) = \left\{  (a_0,a_1,\ldots,a_n) \in \mathbb{N}^{n+1} \ \middle\vert \begin{array}{l}
 \displaystyle\sum_{j=0}^{n} a_j = (n-1)m, \   0<a_i \leq m \text{ for } i=2,\cdots ,n, 
  \\
  rm< a_0 \leq (r+1)m \text{ and }  sm< a_1 \leq (s+1)m 
  \end{array}  \right\}.
\] for $0\leq r+s \leq n-2$	with $r,s\geq 0$. By using the sets $S(r,s)$ and the description of a basis for $H^1(X,\mathcal{O}_X) $ in \cite[Thm.1]{st}, we find an explicit basis $\mathcal{B} = \displaystyle\bigcup_{r+s=0}^{n-2} \mathcal{B}(r,s)$ for $H^1(X,\mathcal{O}_X) $.Therefore, we state the following theorem:
\begin{theorem} \label{basis}
   If we assume above setup, then we have that the union $$ \mathcal{B} = \displaystyle\bigcup_{r+s=0}^{n-2}\mathcal{B}(r,s)$$ is a basis for the cohomology group  $H^1(X,\mathcal{O}_X).$
\end{theorem} We then restrict ourselves $p=2$ and apply Frobenius operator $F^*$ (computed as in \cite[Thm.2]{st}) on the basis specified in Thm.\ref{basis}  to compute $a$-number of $X$ by putting restrictions on $\lambda_0,\lambda_1, \cdots, \lambda_{n-2}$ and find lower bound for $p$-rank of $X$ in Thm.\ref{anum} and Thm.\ref{prank}, respectively.
\par The last section is reserved to observe a relation between $a$-numbers and $p$-ranks of a pair of curves $\left (  X',X \right)$ where $X'$ is the singular generalized Fermat curve $ C^m(1,1,\lambda_2,...,\lambda_{n-2})$ and $X$  is the smooth model of it. In Ex.\ref{ges}, we derive the following relations: $$ \left\{ 
    \begin{array}{lr}
        a(X)=a(X')\\
      \sigma(X)= \sigma(X') - (n-2)^m(m-1)
    \end{array}
\right\}  .$$

	The notation is as follows. \\
	
	 $k$ is algebraically closed field of characteristic $p > 0$. 
	
    $X$ is a projective smooth curve over $k$, of genus $g \ge 2$. 

     $X'$ is a projective singular curve over $k$, of genus $g \ge 2$.
 
	$J_X$ is the Jacobian of $X$. 

        $J_X'$ is the Jacobian of $X'$.

	$\omega_X$ is the canonical bundle of $X$. 
	
	$\mu_p$, $\alpha_p$ denote the usual infinitesimal group schemes. 
	
	If $H$ is a $k$-group scheme, $H[n]$ is the kernel of multiplication by $n$ in $H$. 
	
	$a(H) = \mbox{dim}_k(\operatorname{Hom}_{k-gr}(\alpha_p, H))$.
	
	{$\sigma(H)$ is defined by $p^{\sigma(H)}  = |\operatorname{Hom}_{k-gr}(\mu_p,~H)|$.

		\section{Singular curves}
		
		We let $X'$ be a (singular) integral proper curve over $k$ and $\pi : X \to X'$ be the normalization map. 
		The Jacobian $J_{X'}$ sits in the following exact sequence of $k$-group schemes 
		\begin{equation} \label{eq}
		   0 \longrightarrow G \longrightarrow J_{X'} \longrightarrow J_X \longrightarrow 0,
		\end{equation}
	 where $G$ is an affine connected group.\\
		
		We let $L$ be a line bundle on $X'$ and recall the following basic facts :
		\begin{itemize}
		  \item Duality : Cup product composed with the residue map gives a perfect pairing
			$$H^1(X', L) \times H^0(X', \Omega_{X'} \otimes L^{-1}) \to H^1(X', \Omega_{X'}) \cong  k$$
			\cite[Chap. IV, no. 10, Last Remark]{S1}. 
			\item We have the Riemann-Roch Theorem
			$$\chi(L) = \operatorname{deg}(L) + 1 - p_a(X').$$
	Here	$p_a(X') = 1 - \chi(X')
			= \operatorname{dim}(H^1(X', \mc{O}_{X'}))= ~\operatorname{dim}(H^0(X', \Omega_{X'}))$
			is the arithmetic genus of $X'$ \cite[Chap. IV, no. 6, Thm. 1]{S1}.
                \item By applying duality in the special case $L = \mc{O}_{X'}$ we obtain  
                an isomorphism  $$H^0(X', \Omega_{X'}) \cong H^0(J_{X'}, \Omega_{J_{X'}})$$.
			
		\end{itemize}
		Next, we recall the concept of ordinarity for curves and the definitions of the $p$-rank and the $a$-number.
 \begin{definition} \label{od} \cite[Thm. 2.2]{Su}
     Set $W=H^1(X',\mathcal{O}_X')$. Let $W^{s}$ be the largest $F$ invariant subspace of $W$ and $W^*$ be the subspace of $W$ on which $F$ acts as zero map. The natural numbers $\sigma(X') = \operatorname{dim}_k(W^s)$ and $a(X')= \operatorname{dim}_k(W^*)$ are called the $p$-rank and the $a$-number of $X'$, respectively. We say that $X'$ is an ordinary curve if $ W^s = W$  
		\end{definition}
  
		 The following facts are in  \cite[Thm. 7.1]{Y}:
		 	\begin{itemize}
		
		\item[$\mathbf{a)}$] The $p$-rank of $X'$ coincides with the $p$-rank
		$\sigma(J_{X'})$ of its Jacobian i.e., $\sigma(X')=\sigma(J_{X'})$. 
		\item[$\mathbf{b)}$] The $a$-number of $X'$  coincides with the $a$-number
		$a(J_{X'})$ of its Jacobian i.e., 	$a(X')=a(J_{X'})$. 

	\end{itemize}

 \vskip 0.5cm
 
	  Let $X$ be a non-singular proper integral curve defined over algebraically closed field $k$ of characteristic $p>0$. The $\it{Cartier ~operator}$ $\mathscr{C}$ defined in \cite[Chapter 2, Section 6]{Ca} is a $1/p-$linear operator acting on the sheaf $\Omega_{X/k}$ of differential forms for $X$ which satisfies the following properties:
		   \begin{itemize}
		       	\item[1)] $\mathscr{C}(\omega_1 + \omega_2) =\mathscr{C}(\omega_1) + \mathscr{C}(\omega_2) $
		       		\item[2)] $\mathscr{C}(df) = 0 $
		       	\item[3)] $\mathscr{C}(f^p\omega) = f\mathscr{C}(\omega) $
		       		\item[4)] $\mathscr{C}(f^{p-1}df) = df$
		       			\item[5)] $\mathscr{C}(df/f) = df/f $
		       	
		   \end{itemize}
		   for all local sections $\omega_1$, $\omega_2$ and  $\omega$ (respectively $f$) of $\Omega_X$ (respectively of ${\mc O}_{X}$).\\ In particular, the operator $\mathscr{C}$ acts on $H^0(X, \Omega_{X})$, regular differential forms on $X$.
		\begin{remark}
		     The action of $\mathscr{C}$ can be extended to rational differential forms on $X$.  The following lemma shows that in a normalization set up $\pi : X \to X'$ restricting to the subspace $\pi^*(H^0(X', \Omega_{X'}))$ of rational forms on $X$ we obtain a $1/p-$linear action on the cohomology group $H^0(X', \Omega_{X'})$. 
		\end{remark}
		
		\begin{lemma}\label{cartier}
			Let $\mathscr{C}$ be the Cartier operator acting on the \textit{rational forms}
			on $X$ and \\ $F:H^1(X', {\mc O}_{X'}) \to  H^1(X', {\mc O}_{X'})$ be the Frobenius map. We have 
			\begin{itemize}
			\item[a)] ${\mathscr{C}}(H^0(X', \Omega_{X'})) \subset H^0(X', \Omega_{X'})$.
			\item[b)] The duality pairing
$$\langle~, ~\rangle : H^1(X', {\mc O}_{X'}) \times H^0(X', \Omega_{X'}) \to H^1(X', \Omega_{X'}) \cong k$$
satisfies
$$\langle Fv,\omega \rangle ~= ~\langle v, \mathscr{C}\omega \rangle ^{p}.$$
\end{itemize}
		\end{lemma}
		\begin{proof}
		
		     a) Let $\pi: X\longrightarrow X'$ be the normalization morphism and $t$ be a local parameter at $x\in X$.  $H^0(X', \Omega_{X'})$  consists precisely  of rational differential forms $\omega$ on $X$ which at each
		point $x' \in X'$ satisfy the condition 
		$\displaystyle{\sum_{x_j \mapsto x'}}\mbox{Res}(f\omega, x_j) = 0$ for all $f \in {\mc O}_{X', x'}$ where the sum is taken over all the points $x_j\in X$ such that $\pi(x_j) = x'$.  \\
		Let $\omega = \displaystyle \sum_{n=-m}^{\infty}a_nt^ndt$. Since the Cartier operator $\mathscr{C}$ satisfies the properties 2) - 4), we get  \[ \mathscr{C}(t^ndt) =  \begin{cases} 
      \mathscr{C}(\frac{1}{n+1}dt^{n+1}) = 0 & \text{ if } p\not | n+1 \\
      t^{\frac{n+1}{p}-1}dt & \text{ if } p | n+1
      \end{cases}. 
\]
	Therefore, we conclude, by using semilinearity of the Cartier operator $\mathscr{C}$, that $$\mathscr{C}(\omega) = \displaystyle \sum_{n=-l}^{\infty}a^{1/p}_{pn-1}t^{n-1}dt$$ where $m$ is a non-negative integer and $l$ is the greatest non-negative integer so that $pl + 1 \leq m$.	
	 As $\mbox{Res}(f{\mathscr{C}}(\omega), x) = \mbox{Res}(f^p\omega, x)^{1/p}$ for $x \in X$, we see that for $\omega \in H^0(X', \Omega_{X'})$
	\[ \left(\sum_{x_j \to x'}\mbox{Res}(f{\mathscr{C}}(\omega), x_j)\right)^p =  \sum_{x_j \to x'}\mbox{Res}(f^p\omega, x_j) = 0. \]
		The last equality is due to definition of $\omega$. Hence ${\mathscr{C}}(H^0(X', \Omega_{X'})) \subset H^0(X', \Omega_{X'})$ and the result follows. \\
		
		b) Let $t$ be a local parameter at $x'\in X'$. Let $f$ and $\omega $ be any elements of $H^1(X', {\mc O}_{X'})$ and $ H^0(X', \Omega_{X'})$, respectively. Then $f= \displaystyle \sum_{i=-m}^{\infty}a_{i}t^i$ and $ \omega = \displaystyle\sum_{j=-n}^{\infty}b_{j}t^jdt $ at $x'$ for some non-negative integers $m$ and $n$. For duality between  $H^1(X', {\mc O}_{X'})$ and $ H^0(X', \Omega_{X'})$, we refer \cite[Chap. IV no. 9 and 10]{S1}.  \\
		We need only show that $$\mbox{Res}(f^p\omega, x') = (\mbox{Res}(f{\mathscr{C}}(\omega), x'))^p$$ so that \[ \langle F(f), \omega\rangle = \sum_{x' \in X' }\mbox{Res}(f^p\omega, x') =\left (\sum_{x' \in  X'}\mbox{Res}(f{\mathscr{C}}(\omega), x_j)\right )^p =\langle f, \mathscr{C}(\omega)\rangle^p. \]  We find that  
		$$\mbox{Res}(f^p\omega, x') = \displaystyle\sum_{pi+j=-1}a^p_ib_{j}   $$ and $$\mbox{Res}(f{\mathscr{C}}(\omega), x') = \displaystyle\sum_{i+j=0}a_ib^{1/p}_{pj-1}.   $$ Note that 
		$$\displaystyle\sum_{pi+j=-1}a^p_ib_{j} = \displaystyle\sum_{i}a^p_ib_{-pi-1} = \displaystyle\sum_{i+j=0}a^p_ib_{pj-1} = \displaystyle(\sum_{i+j=0}a_ib^{1/p}_{pj-1})^p.  $$
		Hence we obtain the desired duality.
		\end{proof}

\vskip 0.3cm

In the following Proposition we will use the fact that since 
$k$ is algebraically closed, in the exact sequence (1) the group $G$ decomposes into a product $G = \mathbb{G}_m^r \times G_u$ where $r$ is a non-negative integer and $G_u$ is a unipotent group (that is, a successive extension of the additive group $\mathbb{G}_a$).
	
		\begin{p1} \label{GJ1} In the given setup, the following relations hold:
	\begin{itemize}
	    \item[1)] $\sigma(J_X) = \sigma(J_{X'}) - \sigma(G)
	    = \sigma(J_{X'}) - r$
	    \item[2)] $a(J_X) \ge a(J_{X'}) - a(G) = a(J_{X'}) - a(G_u)$.
	\end{itemize}
		\end{p1}
		\begin{proof}
		
	\begin{itemize}
	\item[1)] 	We apply $\mbox{Hom}_{k-gr}(\mu_p, . )$  to sequence (1) and we get the exact sequence
		$$0 \to \mbox{Hom}_{k-gr}(\mu_p, G) \to\mbox{Hom}_{k-gr}(\mu_p, J_{X'}) \to \mbox{Hom}_{k-gr}(\mu_p, J_X) \to\mbox{Ext}^1(\mu_p, G).$$
		We know that the groups
		$\mbox{Ext}^1(\mu_p, \mathbb{G}_m) ~\mbox{and} ~\mbox{Ext}^1(\mu_p, \mathbb{G}_a)$ are trivial (\cite{O}, p. 81). Therefore, 
		$$\mbox{Ext}^1(\mu_p, G) = \mbox{Ext}^1(\mu_p, \mathbb{G}_m)^r \times \mbox{Ext}^1(\mu_p, G_u) = \mbox{Ext}^1(\mu_p, G_u).$$
		By induction on the chain of successive extensions defining $G_u$, we see that $$\mbox{Ext}^1(\mu_p, G_u)$$
		is trivial. \\
		We compare the orders of the groups in sequence (1) to get $$p^{\sigma(J_{X'})} = p^{\sigma(J_{X})}p^{\sigma(G)}.$$
		Thus we have $\sigma(J_X) = \sigma(J_{X'}) - \sigma(G)$.\\ 
		
		Since $\mbox{Hom}_{k-gr}(\mu_p, \mathbb{G}_a)$ is also trivial,
		again by induction on the chain of successive extensions we find $\mbox{Hom}_{k-gr}(\mu_p, G_u) = 0$. Thus we get
		$$p^{\sigma(G)} = |\mbox{Hom}_{k-gr}(\mu_p, G)| =
		 |\mbox{Hom}_{k-gr}(\mu_p, \mathbb{G}_m)^r|= p^r.$$ 
		Hence ${\sigma(G)} = r$ and  we obtain the the first statement of the Proposition.\\ 
		
	\item[2)] We apply $\mbox{Hom}_{k-gr}(\alpha_p, . )$ to sequence (1) to get the exact sequence 	$$0 \to \mbox{Hom}_{k-gr}(\alpha_p, G) \to \mbox{Hom}_{k-gr}(\alpha_p, J_{X'}) \to\mbox{Hom}_{k-gr}(\alpha_p, J_X) \to\mbox{Ext}^1(\alpha_p, G).$$
	Since $\mbox{Hom}_{k-gr}(\alpha_p, \mathbb{G}_m) ~\mbox{and}  ~ \mbox{Ext}^1(\alpha_p, \mathbb{G}_m)$ are trivial (loc. cit.),
	the exact sequence reduces to $$0 \to \mbox{Hom}_{k-gr}(\alpha_p, G_u) \to \mbox{Hom}_{k-gr}(\alpha_p, J_{X'}) \to \mbox{Hom}_{k-gr}(\alpha_p, J_X) \to \mbox{Ext}^1(\alpha_p, G_u)$$
	and we obtain the relation $a(X) \ge a(J_{X'}) - a(G_u)$.
	
		\end{itemize}
		\end{proof}
	\begin{corollary}\label{cor}  If all singular points of $X'$ are double points of the form $z^2 = x^r, r \ge 3 ~\mbox{odd}$ i.e., its singularities are analytically isomotphic to the singularity at $(0,0)$ of the curve $z^2 = x^r$, then we have $\sigma(J_X) = \sigma(J_{X'})$.
		\end{corollary}
		
		Proof. Under the given hypothesis and with the notation of \cite[Chapter 5, Section 17]{S1} , $R_m= \mathbb{G}_{m,P} \times G_u$. Therefore, $G = R_{m} /\Delta = G_u$. Hence $\sigma(G) = 0$. $\Box$
		
		\vskip 0.1cm
		
			\begin{example}  \label{ex2}
			We will give an example of a singular curve $X'$ such that
		\begin{enumerate}
		\item	$\sigma(J_X) < \sigma(J_{X'})
				~\mbox{and} ~a(J_{X'}) = a(G)$ 
	    \item $X$ is an ordinary curve of genus $g(X) = p_a(X') - ~\mbox{dim}(G)$.
		\end{enumerate}
			
			Let $p = 2$ and consider the plane curve $X'$ \cite[Section 4, Thm. 3]{K} defined by the equation
			$$f(x,y,z)=x^3y^3+ x^3z^3 + y^3z^3 + \lambda z^6 = 0,~\mbox{where} ~\lambda \neq \lambda^2.$$
			The singular curve $X^{'}$ is of arithmetic genus $p_a(X') = 10$ and has precisely two singular points $[1:0:0]$ and $[0:1:0]$
			which are both ordinary threefolds (cf. \cite{K}, Section 4 and see more generally \cite[Chapter I, Section 5, pp. 33-39]{H} for ordinary $r$-fold).
			Thus, we see that the group $G$ in the exact sequence (1) is of dimension 6 with $\mathbb{G}_m^4$ (\cite[Chapter V, Section 17]{S1}) as the multiplicative part. It follows
			that the genus of the normalization $X$  of $X^{'}$ is $g(X) = 4$ and $\sigma(X) = \sigma(X')-4$.
				Now we compute $\sigma(X')$ and $a(X')$ by using 
			the action of the Frobenius map $F$ on 
			$H^1(X^{'}, O_{X^{'}})$
			using the basis \cite[Chapter III, Thm. 5.1]{H}
			$$\{\beta = \dfrac{1}{x^ay^bz^c} \mid a+b+c= 6, \text{ }a,b,c\geq 1 \}$$ for $H^1(X^{'}, O_{X^{'}})$. \\
			Note that $$F(\beta)= f^{p-1}\beta^p= f\beta^2= \dfrac{1}{x^{2a-3}y^{2b-3}z^{2c}}+\dfrac{1}{x^{2a-3}y^{2b}z^{2c-3}}+
			\dfrac{1}{x^{2a}y^{2b-3}z^{2c-3}}+\dfrac{\lambda}{x^{2a}y^{2b}z^{2c-6}}$$
			in $H^1(X^{'}, O_{X^{'}})$ as in ( \cite{H}, Chapter IV, Proposition 4.21) and also see \cite[Proposition 14]{st} for explicit formula for the Frobenius map $F$ on $H^1(X^{'}, O_{X^{'}})$.
			Hence by using explicit basis elements 
			$$\beta_1= \dfrac{1}{xy^2z^3}, ~\beta_2= \dfrac{1}{xy^3z^2},
			~\beta_3= \dfrac{1}{x^2yz^3}, ~\beta_4= \dfrac{1}{x^2y^3z}, ~\beta_5= \dfrac{1}{x^3yz^2}, $$ $$ 
			~\beta_6= \dfrac{1}{x^3y^2z^1}, ~\beta_7= \dfrac{1}{x^2y^2z^2}, ~\beta_8= \dfrac{1}{xyz^4}, ~\beta_9= \dfrac{1}{xy^4z}, ~\beta_{10}= \dfrac{1}{x^4yz},$$
			We get $$F(\beta_1)=\beta_3,  ~F(\beta_2)=\beta_4,~F(\beta_3)=\beta_1,~F(\beta_4)=\beta_2,~F(\beta_5)=\beta_6,~F(\beta_6)=\beta_5,$$
			$$F(\beta_7)=\beta_8+\beta_9 + \beta_{10}, ~F(\beta_8)=\lambda \beta_7,~F(\beta_9)=F(\beta_{10})=0.$$
			It is clear that $F$ acts bijectively on the vector space spanned by the basis $\{\beta_i | i=1,2, \cdots , 8 \}$ and $\operatorname{Rank}([F])=8$. Therefore, $\sigma(X')=8$ and $a(X')=2$ by Definition \ref{od}. Thus, we  see that $\sigma(X) = 4$ (Lemma~\ref{GJ1}) and that $X$ is an ordinary curve and so $a(X)=0$ by Definition \ref{od}. One can also use the formula in Remark \ref{rm2} for affine model of $f$ at $z\neq 0$ to show that $X$ is an ordinary curve.  \hfill{$\Box$}

		\end{example}
		
			\begin{remark} \label{rm2}
		
			We conclude from Lemma~\ref{cartier}b) that instead of working
			with Frobenius acting on $H^1(X', {\mc O}_{X'})$, we could have worked with the Cartier operator on $H^0(X', \Omega_{X'})$. The action of the Cartier operator $\mathscr{C}$ on $H^0(X', \Omega_{X'})$ is given by the following formula (\cite{SV}, Theorem 1.1). $$\mathscr{C}(h\dfrac{dx}{f_y})=(\frac{\partial^{2p-2} }{\partial x^{p-1} \partial y^{p-1}} (f^{p-1}h))^{1/p} \frac{dx}{f_y}, $$
			where $f(x,y)=0$ is the dehomogenization of the equation of $X'$ and $h\in k(X')$.
		
		\end{remark}
		
		\begin{example} \label{ex}
			
			This is an example of a pair $X, X'$ such that the $\sigma(X)=\sigma(X') = 1$. 
			
			We take $p = 7$ and consider the curve  $X^{'} \subset {\mathbb P}^2$ \cite[Section 1]{T} given by the equation
			$$f(x,y,z)= x^5+ y^3z^2+ Axyz^3+ Bxz^4 = 0$$ 
			where $A, B$ are non-zero and $A \neq B$.
			$X^{'}$ is a singular curve of arithmetic genus $p_a(X') = 6$. $X'$ has only one singular point $Q=[0:1:0]$, 
			which is analytically isomorphic to the singularity at $(0,0)$ of the plane curve $z^2 = x^5$ i.e. the comletion of the local ring at $Q$ of $X'$ isomorphic to $k[[x,z]]/(z^2-x^5)$.
			\label{ex1} (\cite{T}, Section I) . Hence, Corollary~\ref{cor} applies and we get $\sigma(X)=\sigma(X')$. \\
			Calculating as in Example~\ref{ex2} by using the basis \cite[Chapter III, Thm. 5.1]{H}	$$\{\beta_1= \dfrac{1}{x^3yz}, ~\beta_2= \dfrac{1}{xy^3z},
			~\beta_3= \dfrac{1}{xyz^3}, ~\beta_4= \dfrac{1}{x^2y^2z}, ~\beta_5= \dfrac{1}{x^2yz^2}, 
			~\beta_6= \dfrac{1}{xy^2z^2}  \}$$
			we find 	$$F(\beta_1)=\beta_3 + 5B\beta_2 ,$$
			$$F(\beta_2)= F(\beta_3)= F(\beta_4) = 0,$$
			$$F(\beta_5)= 5AB^2\beta_1 + 5A^2B\beta_5,$$
			$$F(\beta_6)= 4A^3B\beta_2.$$
			Thus, we have 
			$$F^6(\beta_1)= F^6(\beta_2)= F^6(\beta_3)=  F^6(\beta_4)=F^6(\beta_6) = 0 \text{ and} ~F^6(\beta_5)\neq 0.$$
			It follows that $\sigma(X')=1$ by Definition \ref{od} and we get $\sigma(X)=1$.	\hfill{$\Box$} 	
			
		\end{example}
		
		\begin{remark} \label{rm}
			One can adapt the techniques in Example \ref{ex} to find the $p$-rank of curves on more general surfaces. Let $S$ be a smooth projective surface over an algebraically closed field of positive characteristic $p$ with invariants geometric genus $p_g=0$ and irregularity $q=0$. Let $X$ be a projective curve on $S$ with corresponding divisor $D$. We have the following short exact sequence which defines our curve. 
			$$0\longrightarrow \mathcal{O}_S(-D) \longrightarrow \mathcal{O}_S \longrightarrow \mathcal{O}_X \longrightarrow 0.$$ 
			By using the long exact sequence of cohomology obtained from the above short exact sequence, one sees that 
			$$H^1(X,\mathcal{O}_X) \cong H^2(S, \mathcal{O}_S(-D)) $$ 
   
		\end{remark}

\section{A family of curves}	
		
	\par	We will provide an example of a family of smooth complete intersection curves in ${\mb P}^n$. We will compute the $a$-number of the curves in this family and we will obtain a lower bound on the $p$-rank of these curves by using  the action of Frobenius on cohomology.

 \vspace{0.1cm}
  We will use the following example \cite[Section 2.2]{Hi} of smooth integral complete intersection curves :
 \begin{example} \label{ge}[Generalized Fermat Curve]
     Let $X$ be the curve defined as follows:
     
	$$ C^m(\lambda_0,\lambda_1,...,\lambda_{n-2}) := \begin{Bmatrix}  \lambda_0 x_0^m+x^m_1+x^m_2=0\\ \lambda_1 x_0^m+x^m_1+x^m_3=0 \\ \vdots \\ \lambda_{n-2}x_0^m+x^m_1+x^m_n=0\end{Bmatrix} \subset {\mb P}^n $$
 where $\lambda_0,\lambda_1,...,\lambda_{n-2}$ are pairwise different elements of field $k$ with $\lambda_i \neq 0 \text{ for } i=0,1,...,n-2$. We set $f_i = \lambda_ix_0^m+x^m_1+x^m_{i+2} $ for $i=0,1,...,n-2$.
  \end{example}
We will first prove basic equality in the following to compare the number $\operatorname{dim}_k(H^1(X,\mathcal{O}_X))$ which will be computed in Cor.\ref{dim2} and the cardinality of a basis which we will construct in Thm.\ref{basis}.
\begin{proposition} \label{binom}
    We have the following equality $$ \displaystyle\sum_{i=0}^{t}(-1)^i\binom{t+1-i}{t-i}\binom{n+1}{i} = (-1)^t\binom{n-1}{t}$$ for $n,t\in \mathbb{N} \text{ with } n\geq t+1. $
\end{proposition}
\begin{proof}
    \begin{align*}
 \displaystyle\sum_{i=0}^{t}(-1)^i\binom{t+1-i}{t-i}\binom{n+1}{i} &=  (t+1)\binom{n+1}{0} + \displaystyle\sum_{i=1}^{t}(-1)^i(t+1-i)\binom{n+1}{i}  \\ &= (t+1)\binom{n}{0} + \displaystyle\sum_{i=1}^{t}(-1)^i(t+1-i) 
 \displaystyle \{ \binom{n}{i} + \binom{n}{i-1}  \}  \\
  &= (t+1)\binom{n}{0} + \displaystyle\sum_{i=1}^{t}(-1)^i(t+1-i) 
 \binom{n}{i}  \\& \mathbin{\phantom+} + \displaystyle\sum_{i=1}^{t}(-1)^i(t+1-i)\binom{n}{i-1} \\
  &= (t+1)\binom{n}{0} + \displaystyle\sum_{i=1}^{t}(-1)^i(t+1-i) 
 \binom{n}{i}  \\& \mathbin{\phantom+} + \displaystyle\sum_{i=0}^{t}(-1)^{i+1}(t-i)\binom{n}{i}
\end{align*} 
\begin{align*} 
 & = (t+1-t)\binom{n}{0} + \displaystyle\sum_{i=1}^{t}(-1)^i(t+1-i -(t-i))\binom{n}{i} \\&=  \binom{n}{0}+\displaystyle\sum_{i=0}^{t}(-1)^i\binom{n}{i}  \\&=\binom{n}{0} + \displaystyle\sum_{i=1}^{t}(-1)^i\{ \binom{n-1}{i} + \binom{n-1}{i-1} \}
\\&=  \binom{n-1}{0} + \displaystyle\sum_{i=1}^{t}(-1)^i \binom{n-1}{i} + \displaystyle\sum_{i=1}^{t}(-1)^i\binom{n-1}{i-1} \\&=  \displaystyle\sum_{i=0}^{t}(-1)^i \binom{n-1}{i} + \displaystyle\sum_{i=0}^{t-1}(-1)^{i+1}\binom{n-1}{i}  \\&=  (-1)^t \binom{n-1}{t}.
\end{align*}
\end{proof}
We will now obtain the dimension of $k$-vector space $H^1(X,\mathcal{O}_X)$ as an alternating sum of binomials where $X$ is a complete intersection curve. This result will be crucial for the proof of Thm.\ref{basis}.
 We note that the proof of following result is determined in \cite[Section 2, Thm.1]{AS}by using Hilbert polynomials for projective varieties and in \cite[Section 3, Thm.1]{Kh} by using Newton polyhedra for complete intersection varieties.    
\begin{proposition} \label{dim}
  Let $X_i:f_i=0$ be a degree $m_i$ hypersurface in projective $n$-space
		${\mb P}^n_k$ given by homogenous polynomial $f_i$ of degree $m_i$ for
		$i=1,2,\ldots ,n-1$.Assume the following setup:
  \begin{itemize}
				\item[a)] 
                $Y_i:=X_1\cap X_2 \cap \cdots  \cap X_i$ 
		for $i=1,2,\ldots ,n-1$ and $Y_0= {\mb P}^n_k$ for $n\geq 3,$
				\item[b)] $\operatorname{dim}_k(H^t(Y_{n-t},\mathcal{O}_{Y_{n-t}}(-s))):= h^t(\mathcal{O}_{Y_{n-t}}(-s))$ for $s \geq 0$ and $t=1,2,\cdots , n$,
				\item[c)] $m_0:=0$.
			\end{itemize}

      Then we have the following 
      $$h^t(\mathcal{O}_{Y_{n-t}}(-r)) = \displaystyle\sum_{i=0}^{n-t}(-1)^i\sum_{0=j_0 <j_1 < \cdots < j_{n-t-i} \leq n-t}^{} h^n(\mathcal{O}_{Y_0}(-(m_{j_0} +\cdots + m_{j_{n-t-i}})-r)). $$ 
  \end{proposition}
  \begin{proof}
  
      We will use induction on the natural number $n-t$ as follows: \\
      
      For $n-t =1$, we will use the short exact sequence 
      $$0\longrightarrow \mathcal{O}_{Y_0}(-m_1-r) \longrightarrow \mathcal{O}_{Y_0}(-r) \longrightarrow \mathcal{O}_{Y_1}(-r) \longrightarrow 0.$$ 
      We note that  we have the following (\cite[Section 78, Proposition 5]{S2});
      \begin{itemize}
				\item[a)] $H^j(Y_i,\mathcal{O}_{Y_i}(m))=0$ for $m\leq 0$, $0<j<\operatorname{dim}Y_i=n-i$,
				\item[b)] $H^0(Y_i,\mathcal{O}_{Y_i}(m))=0$ for $m< 0$,
				\item[c)] $H^0(Y_i,\mathcal{O}_{Y_i})=k$,
			\end{itemize}
   for complete intersection varieties. Therefore, we will obtain the following short exact sequence of cohomology groups
   $$0 \longrightarrow H^{n-1}(Y_1,\mathcal{O}_{Y_1}(-r)) \longrightarrow H^n(Y_0,\mathcal{O}_{Y_0}(-m_1-r)) \longrightarrow H^n(Y_0,\mathcal{O}_{Y_0}(-r))\longrightarrow 0.$$
   Hence, we get 
    \begin{align*}
 h^{n-1}(\mathcal{O}_{Y_1}(-r)) &= h^n(\mathcal{O}_{Y_1}(-m_1-r)) - h^n(\mathcal{O}_{Y_1}(-r))   \\ &= \displaystyle\sum_{i=0}^{1}(-1)^i\sum_{0=j_0 <j_1 < \cdots < j_{1-i} \leq 1}^{} h^n(\mathcal{O}_{Y_0}(-(m_{j_0} +\cdots + m_{j_{1-i}})-r)).
\end{align*} 
As $Y_{n-t}$ is the complete intersection of dimension $t$ in $Y_{n-t-1}$ cut out by hypersurface $X_{n-t}:f_{n-t}=0$ of degree $m_{n-t}$, we reach the following short exact sequence as in above,
    \begin{align*}
0 \longrightarrow H^t(Y_{n-t},\mathcal{O}_{Y_{n-t}}(-r)) &\longrightarrow H^{t+1}(Y_{n-t-1},\mathcal{O}_{Y_{n-t-1}}(-m_{n-t}-r)) \longrightarrow  \\ & \longrightarrow H^{t+1}(Y_{n-t-1},\mathcal{O}_{Y_{n-t-1}}(-r))\longrightarrow 0.
\end{align*} 
Hence, by using induction, one sees that 

\begin{align*}
  h^t(\mathcal{O}_{Y_{n-t}}(-r)) &= h^{t+1}(\mathcal{O}_{Y_{n-t-1}}(-m_{n-t}-r)) - h^{t+1}(\mathcal{O}_{Y_{n-t-1}}(-r))   \\ &= \displaystyle\sum_{i=0}^{n-t-1}(-1)^i\sum_{0=j_0 <j_1 < \cdots < j_{n-t-1-i} \leq n-t-1}^{} h^n(\mathcal{O}_{Y_0} (-(m_{j_0} +\cdots + m_{j_{n-t-1-i}})-m_{n-t}-r)) \\&  - \displaystyle\sum_{i=0}^{n-t-1}(-1)^i\sum_{0=j_0 <j_{1} < \cdots < j_{n-t-1-i} \leq n-t-1}^{} h^n(\mathcal{O}_{Y_0}(-(m_{j_0} +\cdots + m_{j_{n-t-1-i}})-r)) \\&=
  h^n(\mathcal{O}_{Y_0} (-(m_{j_1} +\cdots + m_{n-t})-r)) \\&  + \displaystyle\sum_{i=1}^{n-t-1}(-1)^i\sum_{0=j_0 <j_1 < \cdots < j_{n-t-1-i} \leq n-t-1}^{} h^n(\mathcal{O}_{Y_0} (-(m_{j_0} +\cdots + m_{j_{n-t-1-i}})-m_{n-t}-r)) \\&  - \displaystyle\sum_{i=1}^{n-t-1}(-1)^{i-1}\sum_{0=j_0 <j_{1} < \cdots < j_{n-t-i} \leq n-t-1}^{} h^n(\mathcal{O}_{Y_0}(-(m_{j_0} +\cdots + m_{j_{n-t-i}})-r)) \\& +
   (-1)^{n-t}h^n(\mathcal{O}_{Y_0}(-r)).
\end{align*} 

 For $1\leq i \leq n-t-1$, any $i$ elements subset of $\{ m_1,\cdots, m_{n-t}\}$ either contains $m_{n-t}$ or not. Therefore, any subset of $\{ m_1,\cdots, m_{n-t}\}$ of cardinality $i$ can be constructed in such a way that one either chooses $i$ element(s) from the set $\{ m_1,\cdots, m_{n-t-1}\}$ or chooses $i-1$ element(s) from the set $\{ m_1,\cdots, m_{n-t}\}$ and adds $m_{n-t-1}$ in it. Hence, one finds that 
 \begin{align*}
   h^t(\mathcal{O}_{Y_{n-t}}(-r)) &=  h^n(\mathcal{O}_{Y_0} (-(m_1 +\cdots + m_{n-t})-r)) + (-1)^{n-t}h^n(\mathcal{O}_{Y_0}(-r))  \\& + \displaystyle \sum_{i=1}^{n-t-1}(-1)^i\sum_{0=j_0 <j_1 < \cdots < j_{n-t-i} \leq n-t}^{} h^n(\mathcal{O}_{Y_0} (-(m_{j_0} +\cdots + m_{j_{n-t-i}})-r))  \\& = \displaystyle \sum_{i=0}^{n-t}(-1)^i\sum_{0=j_0 <j_1 < \cdots < j_{n-t-i} \leq n-t}^{} h^n(\mathcal{O}_{Y_0} (-(m_{j_0} +\cdots + m_{j_{n-t-i}})-r)).
\end{align*}

  \end{proof}
  \begin{corollary} \label{dim2}
     If a curve $X$ is given as the generalized Fermat curve $ C^m(\lambda_0,\lambda_1,...,\lambda_{n-2})$, then we have $$ \operatorname{dim}_kH^1(X,\mathcal{O}_X) = \displaystyle\sum_{i=0}^{n-1}(-1)^i\binom{n-1}{i}\binom{(n-i-1)m-1}{n}.$$
  \end{corollary}
  \begin{proof}
      As $ n-t = 1 \text{ and }m_1=\cdots = m_{n-1}=m$ for generalized Fermat curves in ${\mb P^n}$, We have the following equality $$ h^1(\mathcal{O}_X) = \displaystyle\sum_{i=0}^{n-1}(-1)^i\binom{n-1}{n-1-i} h^n(\mathcal{O}_{\mb P^n}(-(n-i-1)m)).$$
      However, one has the following description of the cohomology group 
       $$ H^n({\mb P}^n,\mathcal{O}_{{\mb P}^n}(-s)) = \operatorname{Span}_k\left(\{ \dfrac{1}{ x_0^{\alpha_0}x_1^{\alpha_1} \cdots x_n^{\alpha_n}} \mid \displaystyle\sum_{i=0}^{n} \alpha_i = s, \text{ } \alpha_i \geq 1 \}  \right)$$ for $ s\geq n-2$ \cite[Chapter III, Thm. 5.1]{H}. As a result, we conclude that 
      $$ h^1(\mathcal{O}_X) = \displaystyle\sum_{i=0}^{n-1}(-1)^i\binom{n-1}{i}\binom{(n-i-1)m-1}{n}.$$
  \end{proof}
  Our next work is to construct an explicit basis for the $k$-vector space $ H^1(X,\mathcal{O}_X) $ where $X$ is the generalized Fermat curve as in Ex.$\ref{ge}$. Recall that elements $\alpha$ of $H^1(X,\mathcal{O}_X)  $ are described as follows: 
  $$ \alpha \in  H^n({\mb P}^n,\mathcal{O}_{{\mb P}^n}(-(n-1)m)) \text{ with } \alpha f_i = 0 \text{ in } H^n({\mb P}^n,\mathcal{O}_{{\mb P}^n}(-(n-2)m))$$ for $i=0,1,\cdots , n-2$ by Thm. 
  \newpage
  We set $\alpha = x^{-a_0}_0 x^{-a_1}_1  \cdots x^{-a_n}_n$ such that $a_0+a_1 + \cdots + a_n = (n-1)m$ for $a_i \geq 1, i=0,1,\cdots,n+1$ and set $f_i= \lambda_ix_0^m + x_1^m + x^m_{i+2}$ for $i=0,1,\cdots,n-2$. For 
the cohomology class $\alpha $, we have \begin{align*}
 x_0^{-\alpha_0}x_1^{-\alpha_1} \cdots x_n^{-\alpha_n}f_i &= \lambda_i\,x_0^{-\alpha_0+m}\,x_1^{-\alpha_1} \, \cdots \, x_n^{-\alpha_n} \\ & \mathbin{\phantom+}  + \, x_0^{-\alpha_0} \, x_1^{-\alpha_1+m} \,  \cdots \, x_n^{-\alpha_n}  \\
  & \mathbin{\phantom+}  + \, x_0^{-\alpha_0} \, \cdots \,  x_{i+1}^{-\alpha_{i+1}}\, x_{i+2}^{-\alpha_{i+2}+m}\,  x_{i+3}^{-\alpha_{i+3}} \,  \cdots \,  x_n^{-\alpha_n}.  
\end{align*} 
 Hence we now see that 
\begin{align*}
 x_0^{-a_0}x_1^{-a_1} \cdots x_n^{-a_n}f_i &=  \lambda_i\,x_0^{-a_0+m}\,x_1^{-a_1} \, \cdots \, x_n^{-a_n} \, + \, x_0^{-a_0} \, x_1^{-a_1+m} \,  \cdots \, x_n^{-a_n}  \\
  &\mathbin{\phantom+} + \, x_0^{-a_0} \, \cdots \, x_{i+2}^{-a_{i+2}+m} \,  \cdots \,  x_n^{-a_n} \\
  & = 0
 \end{align*}
if and only if $ -a_{i+2}+m \geq 0, \text{ } -a_0+m \geq 0 \text{ and } -a_1+m \geq 0$ for $i=0,1, \ldots , n-2$ if and only if $ a_i \leq m $ for $i= 0,1, \ldots , n$. We define set $S(0,0)$ as 

\[
   S(0,0) = \left\{  (a_0,a_1,\ldots,a_n) \in \mathbb{N}^{n+1} \ \middle\vert \begin{array}{l}
 \displaystyle\sum_{j=0}^{n} \alpha_j = (n-1)m , \  0 < a_i \leq m, \ i=0,1,\cdots,n  \\

  \end{array}  \right\}.
\]
 Note that $ \operatorname{Span}_k( \{ x_0^{-a_0}x_1^{-a_1} \cdots x_n^{-a_n} | (a_0, \cdots, a_n) \in S(0,0) \} ) \subsetneq H^1(X,\mathcal{O}_X)$. Therefore, we will consider $ (a_0,a_1,\ldots,a_n) \in \mathbb{N}^{n+1}$ with either $a_0 > m $ or $a_1 > m $ and produce a basis element for $H^1(X,\mathcal{O}_X) $. Let us define sets $S(r,s)$ for $r+s \neq 0$ as follows:

 \[
   S(r,s) = \left\{  (a_0,a_1,\ldots,a_n) \in \mathbb{N}^{n+1} \ \middle\vert \begin{array}{l}
 \displaystyle\sum_{j=0}^{n} a_j = (n-1)m, \   0<a_i \leq m \text{ for } i=2,\cdots ,n, 
  \\
  rm< a_0 \leq (r+1)m \text{ and }  sm< a_1 \leq (s+1)m 
  \end{array}  \right\}.
\]
For any $ (a_0,a_1,\ldots,a_n) \in S(r,s) $ and $t=0,1, \cdots , n-2$, we set $\beta_r^s(a_0,\cdots,a_n)$ we have 
 \begin{align*}
  x_0^{-a_0}x_1^{-a_1} \cdots x_n^{-a_n}f_t &=  \lambda_tx_0^{-a_0+m}x_1^{-a_1} \cdots x_n^{-a_n} + x_0^{-a_0}x_1^{-a_1+m} \cdots x_n^{-a_n} \\
  &= x_{t+2}^m  (\lambda_tx_0^{-a_0+m}x_1^{-a_1}  \cdots  x_{t+2}^{-a_{t+2}-m}   \cdots   x_n^{-a_n}  \\ &\mathbin{\phantom+} + \, x_0^{-a_0}x_1^{-a_1+m}  \cdots x_{t+2}^{-a_{t+2}-m}   \cdots   x_n^{-a_n}) \\& \neq 0. 
 \end{align*}
 We next consider an element of the form

 \begin{align*}
 x_0^{-a_0}x_1^{-a_1} \cdots x_n^{-a_n} &-  x_0^{-a_0+m}x_1^{-a_1}\left( \sum_{i=2}^{n} \lambda_{i-2}x_i^{-a_i-m}\prod_{j\geq2,j\neq i}^{n} x_j^{-a_r}  \right) \\
  &- x_0^{-a_0}x_1^{-a_1+m}\left( \sum_{i=2}^{n} x_i^{-a_i-m}\prod_{j\geq2,j\neq i}^{n} x_j^{-a_r}  \right). 
 \end{align*}
If one multiplies this element by $f_t$, one gets 
\begin{align*}
  &-\lambda_t x_0^{-a_0+2m}x_1^{-a_1}\left( \sum_{i=2}^{n} \lambda_{i-2}x_i^{-a_i-m}\prod_{j\geq2,j\neq i}^{n} x_j^{-a_r}  \right) - x_0^{-a_0+m}x_1^{-a_1+m}\left( \sum_{i=2}^{n} x_i^{-a_i-m}\prod_{j\geq2,j\neq i}^{n} x_j^{-a_r}  \right) \\
  &- \lambda_t x_0^{-a_0+m}x_1^{-a_1+m}\left( \sum_{i=2}^{n} \lambda_{i-2}x_i^{-a_i-m}\prod_{j\geq2,j\neq i}^{n} x_j^{-a_r}  \right) - x_0^{-a_0}x_1^{-a_1+2m}\left( \sum_{i=2}^{n} x_i^{-a_i-m}\prod_{j\geq2,j\neq i}^{n} x_j^{-a_r}  \right)
\end{align*}
and one produces an element 
\begin{align*}
 x_0^{-a_0}x_1^{-a_1} \cdots x_n^{-a_n} &-  x_0^{-a_0+m}x_1^{-a_1}\left(\sum_{i=2}^{n} \lambda_{i-2}x_i^{-a_i-m}\prod_{j\geq2,j\neq i}^{n} x_j^{-a_j}\right) \\
  &- x_0^{-a_0}x_1^{-a_1+m}\left( \sum_{i=2}^{n} x_i^{-a_i-m}\prod_{j\geq2,j\neq i}^{n} x_j^{-a_j}  \right) \\
  &+ x_0^{-a_0+2m}x_1^{-a_1}\left( \sum_{2\leq i_1<i_2\leq n}^{} \lambda_{i_1-2}\lambda_{i_2-2}x_{i_1}^{-a_{i_1}-m}x_{i_2}^{-a_{i_2}-m}\prod_{j\geq2,j\neq i_1,i_2}^{n} x_j^{-a_j}  \right) \\
  &+ x_0^{-a_0+m}x_1^{-a_1+m}\left( \sum_{2\leq i_1<i_2\leq n}^{} (\lambda_{i_1-2}+\lambda_{i_2-2})x_{i_1}^{-a_{i_1}-m}x_{i_2}^{-a_{i_2}-m}\prod_{j\geq2,j\neq i_1,i_2}^{n} x_j^{-a_j}  \right)  \\
  &+ x_0^{-a_0}x_1^{-a_1+2m}\left( \sum_{2\leq i_1<i_2\leq n}^{} x_{i_1}^{-a_{i_1}-m}x_{i_2}^{-a_{i_2}-m}\prod_{j\geq2,j\neq i_1,i_2}^{n} x_j^{-a_j}  \right).
 \end{align*}
This process must stop after finitely many steps because we have $rm< a_0 \leq (r+1)m \text{ and } $ $  sm< a_1 \leq (s+1)m.$ Therefore, if one follows this process inductively, one reaches a basis element in $H^n({\mb P}^n,\mathcal{O}_{{\mb P}^n}(-(n-2)m))$ of the form
\[  \alpha_r^s(a_0,a_1,\cdots,a_n) =  x_0^{-a_0}x_1^{-a_1} \cdots x_n^{-a_n} + \displaystyle \sum_{l+q=1,l\leq r, q\leq s}^{r+s} (-1)^{l+q}x_0^{-a_0 + lm}x_1^{-a_1+qm} \varphi_l^q\] where 
 \begin{align*}
\text{     } &\varphi_l^q = \displaystyle\sum_{2\leq i_1< \cdots <i_{l+q}\leq n}^{} \beta_l^q( i_1,\cdots,i_{l+q}) x_{i_1}^{-a_{i_1}-m}\cdots x_{i_{l+q}}^{-a_{i_{l+q}}-m}\displaystyle\prod_{j\geq2,j\neq i_1,\cdots,i_{l+q}}^{n} x_j^{-a_j} 
 \end{align*} 
 and, the coefficient $\beta_l^q(i_1,\cdots,i_{l+q}) $ is determined inductively by the coefficients $\beta_{l-1}^q(i'_1,\cdots,i'_{l+q-1}) $ and $\beta_l^{q-1}(i'_1,\cdots,i'_{l+q-1}) $ for $l+q \geq 2$. If we set $\alpha_0^0(a_0,a_1,\cdots,a_n) =  x_0^{-a_0}x_1^{-a_1} \cdots x_n^{-a_n} $ for all $(a_0,a_1,\cdots,a_n) \in S(0,0)$ and define the subset 
\[ 
\mathcal{B}(r,s) = \left \{ \alpha_r^s(a_0,a_1,\cdots,a_n) \mid (a_0,a_1, \cdots, a_n) \in S(r,s) \right \}  
\]
of $H^n({\mb P}^n,\mathcal{O}_{{\mb P}^n}(-(n-2)m))$ for $0\leq r+s \leq n-2$. Hence we state the following Theorem 
\begin{t2}
If we assume above setup, then we have that the union $$ \mathcal{B} = \displaystyle\bigcup_{r+s=0}^{n-2}\mathcal{B}(r,s)$$ is a basis for the cohomology group  $H^1(X,\mathcal{O}_X).$
\end{t2} 
\begin{proof}
For any $\alpha_r^s(a_0,a_1,\cdots,a_n) \in \mathcal{B}(r,s) $, since $ \alpha_r^s(a_0,a_1,\cdots,a_n)  $ is uniquely determined by $(a_0,a_1,\cdots, a_n) \in S(r,s)$, we conclude that $$ \mathcal{B} = \displaystyle\bigcup_{r+s=0}^{n-2}\mathcal{B}(r,s)$$ is linearly independent subset of $ H^n({\mb P}^n,\mathcal{O}_{{\mb P}^n}(-(n-2)m)).$ To finish the proof, we will  show the following:
\begin{itemize}
    \item[a)]  $\mathcal{B} \subset H^1(X,\mathcal{O}_X) $,
    \item[b)] $\operatorname{Card}(\mathcal{B}) = \operatorname{dim}_k( H^1(X,\mathcal{O}_X) )$.
\end{itemize}
We note that $$f_{t-2}(-1)^{r+s}x_0^{-a_0 + rm}x_1^{-a_1+sm} \varphi_r^s = (-1)^{r+s}x_0^{-a_0 + rm}x_1^{-a_1+sm}x_t^m\varphi^s_r$$ for $t\geq 2$ and, by above construction, $ \beta_1^0(t)= \lambda_{t-2}, \, \beta_0^1(t)=1$ are determined by the coefficient $ \beta_0^0=1$ of  $x_0^{-a_0}x_1^{-a_1} \cdots x_n^{-a_n} $. We now assume $ \beta_{r-1}^s(i_1,\cdots,i_{r+s-1})$ and $ \beta_r^{s-1}(i_1,\cdots,i_{r+s-1})$ are uniquely determined by the previous coefficients for any 
$ \{i_1,\cdots,i_{r+s-1}\} \subset \{1,\cdots,n+1\} $. We find
\begin{align*}
x_t^m \varphi_r^s & = x_t^m\displaystyle\sum_{\substack{ 2\leq i_1< \cdots <i_{r+s}\leq n} } \beta_r^s( i_1,\cdots,i_{r+s}) x_{i_1}^{-a_{i_1}-m}\cdots x_{i_{r+s}}^{-a_{i_{r+s}}-m}\displaystyle\prod_{j\geq2,j\neq i_1,\cdots,i_{l+q}}^{n} x_j^{-a_j} \\ &= \displaystyle\sum_{\substack{ 2\leq i'_1< \cdots <i'_{r+s-1}\leq n \\ t\notin  \{i'_1,\cdots, i'_{r+s-1} \}} } \beta_r^s( i'_1,\cdots,t,\cdots ,i'_{r+s-1}) x_{i'_1}^{-a_{i'_1}-m}\cdots x_{i'_{r+s-1}}^{-a_{i'_{r+s-1}}-m}\displaystyle\prod_{j\geq2,j\neq i'_1,\cdots,i'_{r+s-1}}^{n} x_j^{-a_j}.
 \end{align*} 
As 
\begin{align*}
\text{     }&\mathbin{\phantom+} \lambda_{t-2}x_0^m \left (x_0^{-a_0 + (r-1)m}x_1^{-a_1+sm} \varphi_{r-1}^s \right) + x_1^m \left (x_0^{-a_0 + rm}x_1^{-a_1+(s-1)m} \varphi_r^{s-1} \right)
 \\ &= x_0^{-a_0 + rm}x_1^{-a_1+sm}\left ( \lambda_{t-2}\varphi_{r-1}^s + \varphi_r^{s-1} \right),
 \end{align*} 
 we set $$  \beta_r^s( i'_1,\cdots,t,\cdots ,i'_{r+s-1}) = \lambda_{t-2} \beta_{r-1}^s( i'_1,\cdots ,i'_{r+s-1}) + \beta_{r}^{s-1}( i'_1,\cdots ,i'_{r+s-1}) $$ for each $t$. Therefore, the terms of   $$f_{t-2}(-1)^{r+s}x_0^{-a_0 + rm}x_1^{-a_1+sm} \varphi_r^s $$ is cancelled by some part of  $$\lambda_{t-2}x_0^m(-1)^{r+s-1}x_0^{-a_0 + (r-1)m}x_1^{-a_1+sm} \varphi_{r-1}^s $$ and  $$x_1^m(-1)^{r+s-1}x_0^{-a_0 + rm}x_1^{-a_1+(s-1)m} \varphi_r^{s-1}. $$ By using induction on  $(r+s)-(l+q)$, if we assume that the component  $$f_{t-2}(-1)^{l+q}x_0^{-a_0 + lm}x_1^{-a_1+qm} \varphi_l^q $$ of  $\alpha_r^s(a_0,a_1,\cdots,a_n) $ is cancelled by some part of $$\lambda_{t-2}x_0^m(-1)^{l+q-1}x_0^{-a_0 + (l-1)m}x_1^{-a_1+qm} \varphi_{l-1}^q $$ and  $$x_1^m(-1)^{l+q-1}x_0^{-a_0 + lm}x_1^{-a_1+(q-1)m} \varphi_l^{q-1} $$ for $0\leq l+q \leq r+s-2$, we see that 
\begin{align*}
f_{t-2}\alpha_r^s(a_0,a_1,\cdots,a_n) & =  (\lambda_{t-2}x_0^m+x_1^m)x_0^{-a_0}x_1^{-a_1} \cdots x_n^{-a_n}  \\& \mathbin{\phantom+} - x_t^mx_0^{-a_0 + m}x_1^{-a_1} \varphi_1^0 \\& \mathbin{\phantom+} - x_t^mx_0^{-a_0 }x_1^{-a_1+m} \varphi_0^1
 \end{align*} 
 after cancellation by induction. However, one has 
 \begin{align*}
\text{         } & x_t^mx_0^{-a_0 + m}x_1^{-a_1} \varphi_1^0= \lambda_{t-2} x_0^{-a_0 + m}x_1^{-a_1} x_2^{-a_2}\cdots x_n^{-a_n} \\& x_t^mx_0^{-a_0 }x_1^{-a_1+m} \varphi_0^1= x_0^{-a_0 }x_1^{-a_1+m}x_2^{-a_2}\cdots x_n^{-a_n}
 \end{align*} 
 as $\beta_1^0(t)=\lambda_{t-2} \text{ and } \beta_0^1(t)=1 $.  As a result, we reach  $$ f_{t-2}\alpha_r^s(a_0,a_1,\cdots,a_n) = 0 $$   for $2\leq t \leq n+1$. Hence,  $\mathcal{B} \subset H^1(X,\mathcal{O}_X) $. \\ \\ We now compute the cardinality of $\mathcal{B} $:
 $$ \operatorname{Card}(\mathcal{B})= \displaystyle\sum_{\substack{r\geq 0, \, s\geq 0 \\ 0\leq r+s \leq n-3}}^{}\operatorname{Card}(S(r,s))= \displaystyle\sum_{t=0}^{n-3} \binom{t+1}{t}\operatorname{Card}(S(t,0)).$$
 Recall that the set $S(t,0)$ is defined as 
  \[
   S(t,0) = \left\{  (a_0,a_1,\ldots,a_n) \in \mathbb{N}^{n+1} \ \middle\vert \begin{array}{l}
 \displaystyle\sum_{j=0}^{n} a_j = (n-1)m, \   0<a_i \leq m \text{ for } i=1,\cdots ,n, 
  \\
  tm< a_0 \leq (t+1)m  
  \end{array}  \right\}.
\]
Let us compute the cardinality $\operatorname{Card}(S(t,0))$. We are looking for non-negative integer solutions of the problem: \\
 $$ \left\{ 
    \begin{array}{lr}
        \displaystyle\sum_{j=0}^{n} \alpha_j = (n-t-1)m-(n+1) \\
        \alpha_j \leq m-1 \text{ for } j=0,\ldots ,n
    \end{array}
\right\}  $$
 Let $N(t,0)$ be the number of all non-negative integer solutions of \\ $ \displaystyle\sum_{j=0}^{n} \alpha_j = (n-t-1)m-(n+1) $ and $N(t,i)$ be the number of non-negative integer solutions of $ \displaystyle\sum_{j=0}^{n} \alpha_j = (n-t-1)m-(n+1) $ such that at least $i$ of $ \alpha_0, \ldots, \alpha_n$ is greater than or equal to $m$ for $i\geq 1$. Then by the principle of Inclusion-Exclusion we find $ \operatorname{Card}(S(t,0))= \displaystyle \sum_{i=0}^{n-t-2} (-1)^i N(t,i)$ where 
 $$ N(t,i) = \binom{n+1}{i} \operatorname{Card} \left\{ 
    \begin{array}{lr}
        (\alpha_0,\ldots, \alpha_n):\displaystyle\sum_{j=0}^{n} \alpha_j = (n-t-1)m-(n+1)-im
    \end{array}
\right\} $$ for $i=0,\ldots,n-t-2$. Hence $N(t,i)= \binom{n+1}{i} \binom{(n-t-i-1)m-1}{n}$ for $i=0,\ldots,n-t-2$. Therefore, we have
\begin{align*}
\hspace{0.7cm} \operatorname{Card}(\mathcal{B})  &=  \displaystyle\sum_{t=0}^{n-2} \binom{t+1}{t}\operatorname{Card}(S(t,0))  \\& = \displaystyle\sum_{t=0}^{n-2} \binom{t+1}{t}\sum_{i=0}^{n-t-2} (-1)^iN(t,i) \\& = \displaystyle\sum_{t=0}^{n-2} \binom{t+1}{t}\sum_{i=0}^{n-t-2} (-1)^i \binom{n+1}{i} \binom{(n-t-i-1)m-1}{t} \\& = \displaystyle\sum_{t=0}^{n-2} \left[ \sum_{i=0}^{t} (-1)^i \binom{t+1-i}{t-i} \binom{n+1}{i}\right ]\binom{(n-t-1)m-1}{n} \\&= *  
\end{align*}
From Prop.\ref{binom} and by Cor.\ref{dim2},  we see 
\begin{align*}
  \text{     } * &= \displaystyle\sum_{t=0}^{n-2} (-1)^t  \binom{n-1}{t}  \binom{(n-t-1)m-1}{n} \\&=
  \operatorname{dim}_kH^1(X,\mathcal{O}_X).
\end{align*}
 As a result, we complete the proof.
\end{proof}
 
  We now assume that
 $  \text{char}(k) = 2 \text{ and } m\geq 3 \text{ is an odd integer} $ and we will compute the Frobenius map $F^*$ on the set $\mathcal{B}(r,s)$ for each $0\leq r+s \leq n-2$ with $r\geq0, \, s\geq 0.$ First, we analyze vanishing of $F^*$ on $\mathcal{B}(0,0)$  and then see the situation for $\mathcal{B}(r,s)$ for $r+s\geq 1$ by putting some extra conditions on the constants $\lambda_0, \cdots, \lambda_{n-2}$. 
 \vskip 0.3cm
 $\underline{\mathbf{Case1 }}:F^* \, \text{ on } \, \mathcal{B}(0,0)$
  \vskip 0.3cm
 For $(\alpha_0, \ldots, \alpha_n) \in S(0,0)$,
 
 \begin{equation} \label{eq2}
\begin{split}
	F^*( x_0^{-\alpha_0}\cdots x_n^{-\alpha_n}) & =  (f_0\cdots f_{n-2}) x_0^{-2\alpha_0} \cdots x_n^{-2\alpha_n} \\
 & =  \sum_{\substack{\rho \in \operatorname{Sym}(\{0,\ldots,n \}) \\ \rho = (\rho_0\rho_1)(\rho_2\cdots \rho_n)  \\ \text{the sequence } \left\{ \rho_i \right\}^{n}_{i=2}\\ \text{ decreases at most twice} \\ \text{ and if } \rho_i>\rho_{i+1} \text{ for some } i, \\ \text{ then either } \rho_{i+1}=0 \text{ or } \rho_{i+1}=1 } } h_{\rho}  x_{\rho_0}^{-2\alpha_{\rho_0}}x_{\rho_1}^{-2\alpha_{\rho_1}} x_{\rho_2}^{-2\alpha_{\rho_2}+m} \cdots x_{\rho_n}^{-2\alpha_{\rho_n}+m}
	\end{split}
	\end{equation}
 in $H^1(X,\mathcal{O}_X)$ where $h_{\rho} = h_{\rho}(\lambda_0, \ldots , \lambda_{n-2}) \neq 0 $ as \[ h_{\rho}= \begin{cases} 
      \lambda_i & \text{ if } \rho_i = 0 \text{ for some }i\geq2 , \\
      1 & \text{ otherwise}.
   \end{cases}
\] Let us define sets $S_{\rho}= \{\rho_2, \cdots, \rho_n \}$ for each $h_{\rho}$ in the eqn. \ref{eq2}. Note that if $S_\rho = S_{\rho'} = S_{\rho''} $, then we have either $\rho=\rho'$ or $\rho=\rho''$ because the sequence $ \left\{ \rho_i \right\}^{n}_{i=2}$ decreases at most twice to zero or one. If we sum up the coefficients of same (Laurent) monomial, we see that the sum $$\displaystyle\sum_{\rho}^{}h_{\rho}  x_{\rho_0}^{-2\alpha_{\rho_0}}x_{\rho_1}^{-2\alpha_{\rho_1}} x_{\rho_2}^{-2\alpha_{\rho_2}+m} \cdots x_{\rho_n}^{-2\alpha_{\rho_n}+m} $$ becomes $\sum_{}^{}a_Ix^I$ where \[ a_I= \begin{cases} 
      h_{\rho}+ h_{\rho'} & \text{ if } \rho_i = 0 =\rho'_j \text{ for some }i,j\geq2 \text{ with } i\neq j \text{ and } S_\rho = S_{\rho'}, \\ h_{\rho} & \text{ if } \rho_i = 0 \text{ for some }i\geq2 \text{ and }  S_\rho \neq S_{\rho'} \text{ for any } \rho' \neq \rho,
      \\  1 & \text{ otherwise}.
   \end{cases}
\]
As $a_I \neq 0$, there is no cancellation in the sum \ref{eq2}. We observe that $F^*( x_0^{-\alpha_0}\cdots x_n^{-\alpha_n})= 0 $ if and only if at least one of the terms $-2\alpha_{\rho_i} +m \geq 0 $ for $\rho$ and $i=2,\ldots, n$ in each summation of the sum (\ref{eq2}), if and only if at least one of the terms $ \alpha_{\rho_i} \leq \frac{m-1}{2}$ ($m$ is odd) for $\rho$ and $i=2,\ldots, n$ in each summation of the sum (\ref{eq2}). This is the case when at least three of  $ \alpha_{i}$ are less than or equal to $\frac{m-1}{2}$. We define the set $T(0,0) \subset S(0,0)$ as 
  \[
   T(0,0) = \left\{  (a_0,a_1,\ldots,a_n) \in \mathbb{N}^{n+1} \ \middle\vert \begin{array}{l}
 \displaystyle\sum_{j=0}^{n} a_j = (n-1)m, \   0<a_j \leq m \text{ for } j=0,\cdots ,n, 
  \\
  0< \alpha_{j_1},\ldots \alpha_{j_s} \leq (m-1)/2 \text{ for } s\geq 3
  \end{array}  \right\}. \]
Then $\operatorname{Card}(T(0,0)) = \operatorname{Card}(S(0,0)) -\binom{n+1}{n-1}T_{n-1}(0,0) + \binom{n+1}{n}T_{n}(0,0)- \binom{n+1}{n+1}T_{n+1}(0,0) $ where $T_i(0,0)$ is the number of non-negative integer solutions of $ \displaystyle\sum_{j=0}^{n} \alpha_j = (n-1)m-(n+1) $ such that at least $i$ of $ \alpha_0, \ldots, \alpha_n$ is greater than or equal to $(m-1)/2$ for $i\geq n-1$. Therefore 
$$ T_i(0,0) = \binom{n+1}{i} \operatorname{Card} \left\{ 
    \begin{array}{lr}
        (\alpha_0,\ldots, \alpha_n):\displaystyle\sum_{j=0}^{n} \alpha_j = (n-1)m-(n+1)-i(m-1)/2
    \end{array}
\right\}.$$ Hence $T_i(0,0)= \binom{n+1}{i} \binom{(n-1)m-i(m-1)/2-1}{n}$ for $i\geq n-1$.
\vskip 0.3cm
 $\underline{\mathbf{Case2}}:F^* \, \text{ on } \, \mathcal{B}(r,s) \text{ with } r+s\geq 1$
   \vskip 0.3cm 
For any $ \alpha_r^s(a_0,a_1,\cdots,a_n) \in \mathcal{B}(r,s)$ and $f_{t-2} = \lambda_{t-2} x_0^m + x_1^m + x_{t}^m \text{ for } t=2,\cdots,n$, we have 
 \begin{align*}
F^*(\alpha_r^s(a_0,a_1,\cdots,a_n)) & =   \left( \prod_{t=2}^{n} f_{t-2} \right ) (\alpha_r^s(a_0,a_1,\cdots,a_n) )^2   \\& = \left( \prod_{t=2}^{n} f_{t-2} \right) (x_0^{-a_0}x_1^{-a_1} \cdots x_n^{-a_n} )^2 \\&  + \left( \prod_{t=2}^{n} f_{t-2} \right) \left(\sum_{ \substack{l+q=1 }}^{r+s} (-1)^{l+q}x_0^{-a_0 + lm}x_1^{-a_1+qm} \varphi_l^q \right)^2 \\& =0
 \end{align*} 
if and only if we have 
\begin{align}
 &\left( \prod_{t=2}^{n} f_{t-2} \right) (x_0^{- a_0}x_1^{-a_1} \cdots x_n^{-a_n} )^2=0 \label{3.2} \\
 & \left( \prod_{t=2}^{n} f_{t-2} \right) \left(\sum_{ \substack{l+q=1 }}^{r+s} (-1)^{l+q}x_0^{-a_0 + lm}x_1^{-a_1+qm} \varphi_l^q \right)^2 =0 \label{3.3}
\end{align}
because each sum in the equality \eqref{3.3} contains a term of either $x_t^{-2a_t-2m}$ or $x_t^{-2a_t-m}$ for some $t\geq 2$, but any sum in the equality \eqref{3.2} has a term of  
either $x_t^{-2a_t+m}$ or $x_t^{-2a_t}$ for all $t\geq 2$, and so there is no cancellation between the two sums. We now determine conditions on $(a_0, \cdots, a_n)$ so that $$ \left( \prod_{t=2}^{n} f_{t-2} \right) (x_0^{- a_0}x_1^{-a_1} \cdots x_n^{-a_n} )^2=0$$ by putting some restrictions on $ [\lambda_0: \cdots :\lambda_{n-2}] \in {\mb P}^{n-2} $. We assume $\frac{n-2}{2}<r+s \leq n-2$ and so we have $n-1\leq 2r+2s < 2a_0 + 2a_1$. Hence we may find a pair $(l_r,q_s)$ of non-negative integers such that $l_r+q_s=n-1, \, l\leq 2r \text{ and } q\leq 2s$ and define coefficients $A_r^s(\lambda_0, \cdots , \lambda_{n-2})$ for $r\geq 1$ as 
\[ A_r^s(\lambda_0, \cdots , \lambda_{n-2})=\displaystyle\sum_{2\leq i_1<\cdots < i_{l_r}\leq n \,}^{} \displaystyle\prod_{j=1}^{l_r}\lambda_{i_j-2}
\]
and we choose $ A_r^s(\lambda_0, \cdots , \lambda_{n-2})=1 $ for $r=0$.
If we assume $A_r^s(\lambda_0, \cdots , \lambda_{n-2}) \neq 0$, then we have that the Laurent monomial  
$$ A_r^s(\lambda_0, \cdots , \lambda_{n-2})x_0^{-2a_0+l_rm}x_1^{-2a_1+q_sm}
\displaystyle\prod_{j=2}^{n}x_j^{-2a_j} \neq 0 $$ in the sum \eqref{3.2}. Therefore, we obtain $  F^*(\alpha_r^s(a_0,a_1,\cdots,a_n)) \neq 0 $ for $\frac{n-2}{2} < r+s \leq n-2$.
We now assume $1\leq r+s \leq \lfloor\frac{n-2}{2} \rfloor$ and try to solve $F^*(\alpha_r^s(a_0,a_1,\cdots,a_n))=0 $ for $(a_0,\cdots,a_n) \in S(r,s)$. We assume following setup: 
\begin{align*}
  & I _{b+c} = \{ i_1, \cdots, i_{b+c}\} \subset \{ 2, \cdots, n \} \text{ for any }  b\leq 2r+1, \,  c \leq 2s+1 \\ & B_{b}^{c}(\lambda_0, \cdots , \lambda_{n-2} ) = \displaystyle\sum_{
  \{t_1,\cdots, t_b \} \subset I_{b+c} }^{} \, \left ( \prod_{j=1}^{l} \lambda_{t_j-2} \right) \neq 0 \text{ for } b \geq 1  \\& 
 B_{b}^{c}(\lambda_0, \cdots , \lambda_{n-2} ) =1  \text{ for } b=0.
\end{align*}
We next consider the following monomial 
\[
B_{b}^{c} \cdot x_0^{-2a_0 + bm} x_1^{-2a_1 + cm} \left( \prod_{t=1}^{b+c} x_{i_t}^{-2a_{i_t}}\right) \left( \prod_{\substack{j \neq i_t, \, j\geq 2 \\ t=1,\cdots, b+c }}^{n} x_j^{-2a_j + m}\right) 
\]
in \eqref{3.2} and determine whether it is zero or not. Hence, it is zero if and only if we have $-2a_j+m \geq 0$ for some $j\in \{ 2, \cdots , n\} \setminus I_{b+c}$ if and only if $0<a_j < \frac{m}{2}$ for some $j\in \{ 2, \cdots , n\} \setminus I_{b+c}$. If we consider all possible subsets $ I_{b+c} $, then the condition that $ a_j < \frac{m}{2}$ for at least $(n-b-c)$ $j$ values is necessary and sufficient for the equality 
\[
B_{b}^{c} \cdot x_0^{-2a_0 + bm} x_1^{-2a_1 + cm} \left( \prod_{t=1}^{b+c} x_{i_t}^{-2a_{i_t}}\right) \left( \prod_{\substack{j \neq i_t, \, j\geq 2 \\ t=1,\cdots, b+c }}^{n} x_j^{-2a_j + m}\right) =0
\]
for each $I_{b+c} $. For that reason, we define the following sets for the pair $(r,s)$:
 \[
   T_1(r,s) = \left\{  (a_0,a_1,\ldots,a_n) \in \mathbb{N}^{n+1} \ \middle\vert \begin{array}{l}
  rm< a_0 \leq \left(\frac{2r+1}{2}\right)m, \,  sm< a_1 \leq \left(\frac{2s+1}{2}\right)m
 \\ \text{ } \\  0< a_t < \frac{m}{2} \text{ for at least } (n-2r-2s) \text{ } t \text{ values with } t\geq 2
   \\  \displaystyle\sum_{j=0}^{n} a_j = (n-1)m
  \end{array}  \right\},
\]
\[
   T_2(r,s) = \left\{  (a_0,a_1,\ldots,a_n) \in \mathbb{N}^{n+1} \ \middle\vert \begin{array}{l}
  rm< a_0 \leq \left(\frac{2r+1}{2}\right)m, \,  \left(\frac{2s+1}{2}\right)m< a_1 \leq (s+1)m
 \\ \text{ } \\  0< a_t < \frac{m}{2} \text{ for at least } (n-2r-2s-1) \text{ } t \text{ values with } t\geq 2
   \\  \displaystyle\sum_{j=0}^{n} a_j = (n-1)m
  \end{array}  \right\},
\] 
\[
   T_3(r,s) = \left\{  (a_0,a_1,\ldots,a_n) \in \mathbb{N}^{n+1} \ \middle\vert \begin{array}{l}
 \left(\frac{2r+1}{2}\right)m< a_0 \leq (r+1)m, \,  sm< a_1 \leq  \left(\frac{2s+1}{2}\right)m
 \\ \text{ } \\  0< a_t < \frac{m}{2} \text{ for at least } (n-2r-2s-1) \text{ } t \text{ values with } t\geq 2
   \\  \displaystyle\sum_{j=0}^{n} a_j = (n-1)m
  \end{array}  \right\},
\]  
\[
   T_4(r,s) = \left\{  (a_0,a_1,\ldots,a_n) \in \mathbb{N}^{n+1} \ \middle\vert \begin{array}{l}
 \left(\frac{2r+1}{2}\right)m< a_0 \leq (r+1)m, \,  \left(\frac{2s+1}{2}\right)m< a_1 \leq (s+1)m
 \\ \text{ } \\  0< a_t < \frac{m}{2} \text{ for at least } (n-2r-2s-2) \text{ } t \text{ values with } t\geq 2
   \\  \displaystyle\sum_{j=0}^{n} a_j = (n-1)m
  \end{array}  \right\}.
\]  

For $ (a_0,a_1,\ldots,a_n) \in T(r,s):= \displaystyle\bigcup_{i=1}^{4} T^i(r,s)$, by above setup, we have 
 
\[
B_{b}^{c} \cdot x_0^{-2a_0 + bm} x_1^{-2a_1 + cm} \left( \prod_{t=1}^{b+c} x_{i_t}^{-2a_{i_t}}\right) \left( \prod_{\substack{j \neq i_t, \, j\geq 2 \\ t=1,\cdots, b+c }}^{n} x_j^{-2a_j + m}\right) =0
\]
 and also for any $b'< b, \, c'< c$ we have
\[
B_{b'}^{c'} \cdot x_0^{-2a_0 + b'm} x_1^{-2a_1 + c'm} \left( \prod_{t=1}^{b'+c'} x_{i_t}^{-2a_{i_t}}\right) \left( \prod_{\substack{j \neq i_t, \, j\geq 2 \\ t=1,\cdots, b'+c' }}^{n} x_j^{-2a_j + m}\right) =0
\]
 as the part  $ \displaystyle\prod_{\substack{j \neq i_t, \, j\geq 2 \\ t=1,\cdots, b'+c' }}^{n} x_j^{-2a_j + m} $ contains $n-b'-c'$ terms $ x_j^{-2a_j + m}$ with $n-b'-c'> n-b-c$.
 Moreover, we note that the part $x_0^{-2a_0 + 2lm}x_1^{-2a_1+2qm} (\varphi_l^q)^2 $ is allowed to multiplied by 
 $x_0^{(b- 2l)m}x_1^{(c-2q)m}$, otherwise it is zero. Therefore, in such a product, one has at least $n-1-(b+c-2l-2q)$ terms $x_j^{m}$ for $j\geq 2$ and since $ n-1-b-c +2l+2q>l+q$ where 
  $$\varphi_l^q = \displaystyle\sum_{2\leq i_1< \cdots <i_{l+q}\leq n}^{} \beta_l^q( i_1,\cdots,i_{l+q}) x_{i_1}^{-a_{i_1}-m}\cdots x_{i_{l+q}}^{-a_{i_{l+q}}-m}\displaystyle\prod_{j\geq2,j\neq i_1,\cdots,i_{l+q}}^{n} x_j^{-a_j}, $$
we have
  
 $$\left( \prod_{t=2}^{n} f_{t-2} \right) \left(\sum_{ \substack{l+q=1 }}^{r+s} (-1)^{l+q}x_0^{-a_0 + lm}x_1^{-a_1+qm} \varphi_l^q \right)^2 =0 .$$

 Therefore, for any $(a_0,a_1, \cdots, a_n) \in S(r,s)$, $$ F^*(\alpha_r^s(a_0,a_1, \cdots, a_n))=0 \text{ if and only if }(a_0,a_1, \cdots, a_n) \in T(r,s). $$
 As a result we derive the following Theorem:
 \begin{theorem} \label{anum}
    Let $X$ be the generalized Fermat $C^m(\lambda_0,\lambda_1,...,\lambda_{n-2}) $ in $ {\mb P}^n$ over algebraically closed field of characteristic $p=2$. We obtain $a$-number  $a(X)$ of $X$  as 
    $$a(X)= \displaystyle \sum_{\substack{r+s=0\\ r\geq 0, \, s\geq 0}}^{\lfloor \frac{n-2}{2} \rfloor} \operatorname{Card}(T(r,s)).$$ where $[\lambda_0: \cdots :\lambda_{n-2}] \in {\mb P}^{n-2} \setminus V\left( \displaystyle\prod_{\substack{r+s= \lceil \frac{n-1}{2} \rceil \\ r\geq0, \, s\geq 0 }}^{n-2}A_r^s \cdot \displaystyle\prod_{\substack{r+s= 1 \\ r\geq0, \, s\geq 0 }}^{\lfloor \frac{n-2}{2} \rfloor }B_r^s  \right)$.
 \end{theorem}
We will finalize this section by finding a lower bound for the $p$-rank $\sigma(X)$ where $X$ is the smooth generalized Fermat curve over algebraically closed field $k$ of characteristic $p=2$ in $ {\mb P}^n $. Recall that the curve $X= C^m(\lambda_1 \lambda_2,...,\lambda_{n-2})$ is defined as

	$$ C^m(\lambda_0,\lambda_1,...,\lambda_{n-2}) := \begin{Bmatrix}  \lambda_0 x_0^m+x^m_1+x^m_2=0\\ \lambda_1 x_0^m+x^m_1+x^m_3=0 \\ \vdots \\ \lambda_{n-2}x_0^m+x^m_1+x^m_n=0\end{Bmatrix} \subset {\mb P}^n $$
where $\lambda_0,\lambda_1,...,\lambda_{n-2}$ are pairwise different elements of field $k$ with $\lambda_i \neq 0 \text{ for } i=0,1,...,n-2$. We set $f_i = \lambda_ix_0^m+x^m_1+x^m_{i+2} $ for $i=0,1,...,n-2$. Using this curve, we define other generalized Fermat curves in lower dimensional projective spaces as 
$$ C^m(\lambda_{i_0},\lambda_{i_1} \cdots, \lambda_{i_t-2}):= \begin{Bmatrix}  \lambda_{i_0} x_0^m+x^m_1+x^m_{i_0+2}=0\\ \lambda_{i_1} x_0^m+x^m_1+x^m_{i_1+2}=0 \\ \vdots \\ \lambda_{i_t-2} x_0^m+x^m_1+x^m_{i_t}=0 \end{Bmatrix} \subset {\mb P}^t $$
for any $1<t<n$ and for any $t-1$ length subset $\{ \lambda_{i_0},\lambda_{i_1} \cdots, \lambda_{i_t-2}\}$ of $ \{ \lambda_0,\lambda_1,...,\lambda_{n-2}\}$.
As a result, we state the following Theorem:
\begin{theorem} \label{prank}
 Let $X$ be the generalized Fermat curve $C^m(\lambda_0,\lambda_1,...,\lambda_{n-2}) $ in $ {\mb P}^n $. We have the following inequality between the $p$-rank of generalized Fermat curves  $$\sigma(X) \geq \displaystyle\sum_{t=2}^{n-1} (-1)^t \displaystyle \sum_{0\leq i_0 < \cdots <i_t\leq n-2}^{}\sigma(C^m(\lambda_{i_0},\lambda_{i_1} \cdots, \lambda_{i_t-2})) $$ 
 for $p=2$.
\end{theorem}

\begin{proof}
   For any $t-1$ length subset $\{ \lambda_{i_0},\lambda_{i_1} \cdots, \lambda_{i_t-2}\}$ of $ \{ \lambda_0,\lambda_1,...,\lambda_{n-2}\}$ and for any $(a'_{i_0},\cdots, a'_{i_t}) \in S'(r,s) $, set similar to $S(r,s)$ and defined for $C^m(\lambda_{i_0},\lambda_{i_1} \cdots, \lambda_{i_t-2}) $ , we consider elements $(a_{0},\cdots, a_{n}) \in S(r,s) $ of the form \[ a_j= \begin{cases} 
      a'_{i_j} & \text{ if } j= 0,1, \\  a'_{i_q} & \text{ if } j\geq 2 \text{ and }, j=i_q \text{ for some } q\in \{2,\cdots, t\},
      \\  m & \text{ otherwise}.
   \end{cases}
\]
We assume $(F^*)^l(\alpha_r^s(a'_{i_0},\cdots, a'_{i_t}))\neq 0$ for all $l \geq 1$, i.e., 
$ \alpha_r^s(a'_{i_0},\cdots, a'_{i_t})$ counts for the $p$-rank of $ C^m(\lambda_{i_0},\lambda_{i_1} \cdots, \lambda_{i_t-2})  $ and so $ (F^*)^l(\alpha_r^s(a'_{i_0},\cdots, a'_{i_t}))= \displaystyle\sum_{j}^{} \beta_j \alpha_{r_j}^{s_j}(a^{l_j}_{i_0},\cdots, a^{l_j}_{i_t})$, a finite linear combination of basis elements for the curve $ C^m(\lambda_{i_0},\lambda_{i_1} \cdots, \lambda_{i_t-2})$. We will show that $$ (F^*)^l(\alpha_r^s(a_{0},\cdots, a_{n}))= \displaystyle\sum_{j}^{} \beta_j \alpha_{r_j}^{s_j}(a^{l_j}_{0},\cdots, a^{l_j}_{n}) + (\text{ other part }) \neq 0 $$ for each $l$ and so the $ t+1$ tuple $(a'_{i_0},\cdots, a'_{i_t}) \in S'(r,s) $ counts for the $p$-rank \\
$\sigma(C^m(\lambda_0,\lambda_1,...,\lambda_{n-2}))$. We have 
\begin{align}
   F^*(\alpha_r^s(a_0,a_1,\cdots,a_n)) & =   \left( \prod_{j=0}^{n-2} f_{j} \right ) (\alpha_r^s(a_0,a_1,\cdots,a_n) )^2 \nonumber \\&= \left( \prod_{q=0}^{t} f_{i_q} \right )  \left( \prod_{\substack{j=2, \, j \neq i_q \\ 0\leq q \leq t}}^{n} f_{j-2} \right ) (\alpha_r^s(a_0,a_1,\cdots,a_n) )^2 \nonumber  \\& = \left( \prod_{q=0}^{t} f_{i_q} \right )  \left( \prod_{\substack{j=2, \, j \neq i_q \\ 0\leq q \leq t}}^{n} x^m_{j} \right ) (\alpha_r^s(a_0,a_1,\cdots,a_n) )^2 \label{3.4} \\&  +  \left( \prod_{q=0}^{t} f_{i_q} \right )  \left( \text{ other product} \right ) (\alpha_r^s(a_0,a_1,\cdots,a_n) )^2 \label{3.5} \\&= \displaystyle\sum_{j}^{} \beta_j \alpha_{r_j}^{s_j}(a^{1_j}_{0},\cdots, a^{1_j}_{n}) + (\text{ other part }). \nonumber
\end{align} 
We have a term of the form $x_j^{-m}$ for each $j$ with $j\neq i_q$ in each sum of \eqref{3.4}, but there is at least one term of the form $ x_j^{-cm}, \, c\geq 2$ for some $j$ with  $j\neq i_q$ in each sum of \eqref{3.5}. Therefore, there is no cancellation between the terms of \eqref{3.4} and \eqref{3.5}. As a result, we obtain 
$$ (F^*)^l(\alpha_r^s(a_{0},\cdots, a_{n}))= \displaystyle\sum_{j}^{} \beta_j \alpha_{r_j}^{s_j}(a^{1_j}_{0},\cdots, a^{1_j}_{n}) + (\text{ other part }) \neq 0 $$ and if we apply the Frobenius map $F^*$ iteratively, infact we see that $$ (F^*)^l(\alpha_r^s(a_{0},\cdots, a_{n}))= \displaystyle\sum_{j}^{} \beta_j \alpha_{r_j}^{s_j}(a^{l_j}_{0},\cdots, a^{l_j}_{n}) + (\text{ other part }) \neq 0 $$ for each $l$. Hence, by using Inclusion-Exclusion principle, we show that $$\sigma(X) \geq \displaystyle\sum_{t=2}^{n-1} (-1)^t \displaystyle \sum_{0\leq i_0 < \cdots <i_t\leq n-2}^{}\sigma(C^m(\lambda_{i_0},\lambda_{i_1} \cdots, \lambda_{i_t-2})). $$ 
\end{proof}

\vskip 0.5cm
$$ {\mb P}^n $$
\section{A family of singular curves}
 Now we will work on singular generalized Fermat curves. In example\ref{ges}, there is a family of singular curves on which we determine  the $p$-rank and the $a$-number of the smooth model of any given curve in the family in terms of the corresponding invariants of the given curve. 
\begin{example}[Singular Generalized Fermat Curve]\label{ges}
Let $X'$ be the singular curve  \\
$ C^m(1,1, \lambda_2,...,\lambda_{n-2})$ $ \lambda_i \neq 0,1$ for $i= 2,\ldots, n-2$ and $X$ be the smooth model of $X'$. Let $ p$ be the characteristic of field $k$ with $(p,m)=1$. In this example we will observe relations between $p$-ranks $\sigma(X),\sigma(X') $ and between $a$-numbers $a(X),a(X')$. Let $\gamma$ be the principal $2m$-th root of unity and $ \zeta $ be the principal $m$-th root of unity. Set $\mu_j= \gamma \zeta^j$ for $j=1,\ldots,m$. Recall that $P$ is a singular point of $X'$ if and only if $ \operatorname{Rank}(\frac{\partial f_i}{\partial x_j}(P)) < n-1$. This is the case when
$$P=[1:\mu_i:0:0:(\lambda_2 -1)^{1/m}\mu_{\lambda(2)}: \cdots :(\lambda_{n-2} -1)^{1/m}\mu_{\lambda(n-2)}]$$ where $i=1,\ldots,m$ and $\lambda(j)=1,\ldots,m$ for $j=2,\ldots, n-2$.
Let $\widehat{\mathcal{O}}_P$ be the completion of the local ring $\mathcal{O}_P$. In the completion ring $\widehat{\mathcal{O}}_P$, one  may write the polynomials 
$$\begin{Bmatrix} g_0=1+y^m_1+y^m_2\\  g_1= 1+y^m_1+y^m_3 \\ g_2= \lambda_2 +y^m_1+y^m_4 \\ \vdots \\ g_{n-2} = \lambda_{n-2} +y^m_1+y^m_n\end{Bmatrix}$$
as
$$\begin{Bmatrix} g_0= (1-\mu_1y_1)h_1+y^m_2\\  g_1=(1-\mu_1y_1)h_1+y^m_3 \\ g_2= \lambda_2 + (y_4 -\mu_1y_1)h_2 \\ \vdots \\ g_{n-2} =\lambda_{n-2}+ (y_n -\mu_1y_1)h_{n-2} \end{Bmatrix}$$
Where $g_i$ is dehomogenization of $f_i$ for $i= 0,\ldots,n-2$ and $h_j(P) \neq 0$ for $j= 1,\ldots, n-2$. Therefore, we realize the isomorphism 
$$ \widehat{\mathcal{O}}_P \cong k[[ x,y]]/(x^m-y^m).$$
As the polynomial $x^m-y^m$ contains $m$ distinct linear factors, each singularity at $P$ contributes exactly $m$ factors of $\mathbb{G}_m$ in linear algebraic group $G$ attached to the short exact sequence 
$$  0 \longrightarrow G \longrightarrow J_{X'} \longrightarrow J_X \longrightarrow 0 .$$ 
Therefore, one has $G= \mathbb{G}_m^{(n-2)^m(m-1)}$ (\cite[Chapter V, Section 17]{S1}). As a result, we obtain the relations 
$$ \left\{ 
    \begin{array}{lr}
        a(X)=a(X')\\
      \sigma(X)= \sigma(X') - (n-2)^m(m-1)
    \end{array}
\right\}  .$$
\hfill{$\Box$}
 \end{example}

	\end{document}